\documentclass[11pt, reqno]{amsart}
\usepackage{amscd,amsmath,amsthm,amssymb,graphics}
\usepackage{amsfonts,amscd,enumitem,verbatim}
\usepackage[a4paper,top=3cm,left=3cm,right=3cm]{geometry}
\usepackage{xcolor}
\usepackage{float}
\theoremstyle{plain}
\usepackage{booktabs}
\usepackage{hyperref}
\newtheorem{Theorem}{Theorem}
\newtheorem{Lemma}[Theorem]{Lemma}

\newtheorem{Proposition}[Theorem]{Proposition}
\newtheorem{Problem}[Theorem]{Problem}

\theoremstyle{remark}
\newtheorem{remark}[Theorem]{Remark}

\newtheorem{Example}[Theorem]{Example}

\usepackage{enumitem}
\setlist[enumerate,1]{label=\alph*)}

\newtheorem{innercustomgeneric}{\customgenericname}
\providecommand{\customgenericname}{}
\newcommand{\newcustomtheorem}[2]{%
  \newenvironment{#1}[1]
  {%
   \renewcommand\customgenericname{#2}%
   \renewcommand\theinnercustomgeneric{##1}%
   \innercustomgeneric
  }
  {\endinnercustomgeneric}
}

\newcustomtheorem{customthm}{Theorem}

\title[Transcendence of continued fractions over function fields]{Transcendence of continued fractions over function fields and a quantitative version of Uchiyama's theorem}
\author{Federico Accossato}
\address{Dipartimento di Scienze Matematiche, Politecnico di Torino, Corso Duca degli Abruzzi 24, 10129, Torino, Italy}
\email{federico.accossato@polito.it}

\author{Nadir Murru}
\address{Dipartimento di Matematica, Università di Trento, Via Sommarive 14, 38123, Povo (TN), Italy}
\email{nadir.murru@unitn.it}

\author{Giuliano Romeo}
\address{Dipartimento di Scienze Matematiche, Politecnico di Torino, Corso Duca degli Abruzzi 24, 10129, Torino, Italy}
\email{giuliano.romeo@polito.it}

\author{Giulia Salvatori}
\address{Dipartimento di Scienze Matematiche, Politecnico di Torino, Corso Duca degli Abruzzi 24, 10129, Torino, Italy}
\email{giulia.salvatori@polito.it}

\thanks{The authors are members of GNSAGA of INdAM}

\date{}

\begin{document}

\begin{abstract}
Given a field $K$, let $K((T^{-1}))$ be the field of formal power series. Continued fractions in $K((T^{-1}))$ can be defined by analogy with classical real continued fractions and have been widely studied. Some results establish the transcendence of elements of $K((T^{-1}))$ arising from special families of continued fractions, but much remains to be explored. In this paper, assuming that $K$ has characteristic zero, we improve the known analogues of the Maillet--Baker criteria for quasi-periodic continued fractions. A central tool that we prove is a quantitative version of Uchiyama's analogue of Roth's theorem in function fields, which gives an explicit bound for the number of exceptionally good rational approximations to an algebraic power series. This quantitative estimate also yields a Davenport--Roth-type upper bound on the growth of the denominators of the convergents of algebraic elements. Finally, we prove that palindromic continued fractions are either quadratic or transcendental, as in the real case, but using a different proof strategy.
\end{abstract}

\maketitle

\section{Introduction}
Given a field $K$, let $K[T]$ be the ring of polynomials with coefficients in $K$, $K(T)$ its field of fractions and $K((T^{-1}))$ the completion of $K(T)$ with respect to the ultrametric norm
\[ \left| \cfrac{a}{b} \right| = \mu^{\deg a - \deg b} \]
for all non-zero $a/b \in K(T)$, where $\mu$ a fixed real number greater than $1$. In this paper, $\log$ denotes the logarithm with base $\mu$.

Continued fractions over $K((T^{-1}))$ can be introduced in a way analogous to their well-known definition over the real numbers. Specifically, given $\alpha_0 = \sum_{i \leq k} b_i T^i \in K((T^{-1}))$, with $k\in\mathbb Z$, $b_i \in K$ and $b_k \ne 0$, we define the following floor function:
\[ \begin{cases}  \lfloor \alpha_0 \rfloor = \sum\limits_{i=0}^k b_i T^i \in K[T] \quad &\text{if $k \geq 0$}, \\ \lfloor \alpha_0 \rfloor = 0 \quad &\text{otherwise}. \end{cases} \]
Then, the continued fraction expansion $\alpha_0 = [a_0, a_1, \ldots]$ is obtained, for $i = 0, 1, \ldots$, by the usual recursive algorithm
\[ \begin{cases} a_i = \lfloor \alpha_i \rfloor \\ \alpha_{i+1} = \cfrac{1}{\alpha_i - a_i} \end{cases} \]
 where if $\alpha_i=a_i$ for some $i$, then $\alpha_{i+1}$ is not defined and $\alpha=[a_0,\ldots,a_i]$ has a finite continued fraction expansion (see \cite{Las20, Sch00} for surveys on this topic). Many studies on continued fractions over $K((T^{-1}))$ aim to replicate the numerous results known over $\mathbb R$. For instance, there are results concerning quadratic irrationals, which have periodic expansions when $K$ is an algebraic extension of a finite field, whereas this is not true in the other cases. 
It is well known that classical continued fractions are widely used in Diophantine approximation and that several transcendence criteria for continued fractions have been established.
Maillet \cite{Mai} and Baker \cite{Bak} proved that quasi-periodic continued fractions converge to transcendental numbers under certain conditions.
We recall that $[a_0, a_1, \ldots]$ is a quasi-periodic continued fraction if 
\begin{equation*}
a_{m+r_i}=a_m, \quad \text{for} \quad n_i\leq m\leq n_i+(\lambda_i-1)r_i-1,
\end{equation*}
for sequences of positive integers $(\lambda_i)_{i\geq 0}$, $(r_i)_{i\geq 0}$ and $(n_i)_{i\geq 0}$, where $(n_i)_{i\geq 0}$ is strictly increasing and $\lambda_i\geq 2$ for all $i$. 

\begin{Example}\label{ex: quasiperiodic}
Let $n_0=1$, $\lambda_i=n_i^2+1$, $n_{i+1}=n_i + \lambda_i$ for all $i\ge 0$. Let $\alpha \in K((T^{-1}))$, $\alpha=[0,a_1,a_2, \ldots]$ where
\[a_m = \begin{cases}
T &\text{if } n_i \le m < n_{i+1}, \quad i \equiv 0 \pmod{2}, \\
T+1 &\text{if } n_i \le m < n_{i+1}, \quad i \equiv 1 \pmod{2},
\end{cases}\]
that is,
\[\alpha = [0, \underbrace{T,T}_{2}, \underbrace{T+1, \ldots, T+1}_{10}, \underbrace{T, \ldots, T}_{170}, \underbrace{T+1, \ldots, T+1}_{33490}, \ldots].\]
This is a quasi-periodic continued fraction with $r_i=1$.
\end{Example}

We summarize in the following theorem some results due to Baker \cite{Bak}.
\begin{Theorem}[\cite{Bak}]
Let $\alpha = [a_0, a_1, \ldots]$ be a quasi-periodic continued fraction in $\mathbb R$.
\begin{enumerate}
\item If there is a positive constant $C$ such that $r_i < C n_i$ and
\begin{equation*} 
\lim_{i \rightarrow +\infty} \cfrac{\log \lambda_i \sqrt{\log n_i}}{n_i} = +\infty. 
\end{equation*}
then $\alpha$ is either transcendental or quadratic.
\item If the sequence $(a_i)_{i \geq 0}$ is bounded, i.e., $a_i\leq M$ for some $M\geq1$, and 
\[ \limsup_{i \rightarrow + \infty} \frac{\lambda_i}{n_i} > A, \]
where $A$ is a constant depending on $M$, then $\alpha$ is either transcendental or quadratic.
\end{enumerate}
\end{Theorem}
Observe that in the previous theorem, the case in which $\alpha$ is quadratic can be excluded by assuming that the continued fraction is non-periodic.
These results were later improved and generalized by several authors \cite{AB05, AB07, AB07b, ADQZ, Dav, Hone}. Remarkably, in \cite{AB07}, Adamczewski and Bugeaud were able to relax the boundedness condition on the partial quotients in the second part of the previous theorem, requiring only that the sequence $(q_i^{1/i})_{i \geq 1}$ is bounded (where $q_i$ denotes the denominator of the $i$-th convergent) and that
\[ \limsup_{i \rightarrow + \infty} \frac{\lambda_i}{n_i} > 0. \]

Continued fractions have also been used to establish transcendence criteria in the context of $p$-adic numbers (see, e.g., \cite{Kal,Lon,Oot,Bel,ALD} and \cite[Section~8]{Rom}) and for multidimensional continued fractions over the real numbers (see, e.g., \cite{Acc}).

In the case of continued fractions over fields of power series, some results on quasi-periodic continued fractions have been obtained \cite{ADH1, ADH2, HMT, Mka}, but much remains to be explored. In the next theorem, we recall the main results proved in \cite{ADH2} by Ammous, Driss, and Hbaib. 
\begin{Theorem}[\cite{ADH2}]
Let $K$ be a finite field, and let $\alpha = [a_0, a_1, \ldots]$ be a quasi-periodic continued fraction in $K((T^{-1}))$.
\begin{enumerate}
\item If there is a positive constant $C$ such that $r_i < C n_i$ and
\begin{equation*}
 \limsup_{i \rightarrow +\infty} \cfrac{\log \lambda_i}{n_i} = + \infty, \end{equation*}
then $\alpha$ is either transcendental or quadratic.
\item If the sequence $(|a_i|)_{i \geq 0}$ is bounded and 
\[ \limsup_{i \rightarrow + \infty} \frac{\lambda_i}{n_i} = +\infty, \]
then $\alpha$ is either transcendental or quadratic.
\end{enumerate}
\end{Theorem}
As in the real case, one can exclude the case where $\alpha$ is quadratic by assuming non-periodicity. Moreover, these results extend to characteristic zero, although in that setting there are no direct assumptions to rule out the quadratic case. Indeed, only some quadratic irrationals can be characterized by quasi-periodic continued fractions (see, e.g., \cite[Proposition~2.1.4]{Mala}). Note that the notion of quasi-periodicity used therein differs from the one adopted in the present paper.

In this paper, we improve these results as summarized in the following theorem.

\begin{Theorem}\label{Thm: mainquasiperiodic}
Consider $\alpha \in K((T^{-1}))$, where $K$ is a field of characteristic zero. Suppose that $\alpha = [a_0, a_1, \ldots]$ is a quasi-periodic continued fraction. 
\begin{enumerate}
\item If there is a positive constant $C$ such that $r_{i}<Cn_{i}$ and
\begin{equation}\label{Eq: HPquasiperA}
\limsup_{i\rightarrow+\infty}\frac{\log \lambda_{i}\sqrt{\log n_i}}{n_{i}}=+\infty,
\end{equation}
then $\alpha$ is either transcendental or quadratic.
\item If the sequence $(|q_i|^{1/i})_{i \geq 1}$ is bounded and
\[\limsup_{i \rightarrow + \infty} \frac{\lambda_i}{n_i} = + \infty,\]
then $\alpha$ is either transcendental or quadratic.
\end{enumerate}
\end{Theorem}
As a consequence of this result, the continued fraction of $\alpha$ in Example~\ref{ex: quasiperiodic} is transcendental. Indeed, using that notation we have \(|q_i|^{1/i} = \mu\)
for all $i \ge 1$ and
\[\lim_{i \to + \infty }\frac{\lambda_i}{n_i} = \lim_{i \to + \infty }n_i + \frac{1}{n_i} = + \infty,\]
and so the hypotheses of Theorem~\ref{Thm: mainquasiperiodic}~b) are satisfied.

We also prove the following additional main result, which is new in the setting of function fields. It is a transcendence criterion for palindromic continued fractions, which represents a well-known class of transcendental continued fractions in the real setting \cite{AB07b}.

\begin{Theorem}\label{Thm: pal}
Consider $\alpha \in K((T^{-1}))$, where $K$ is a field of characteristic zero. If $\alpha = [0, a_1, a_2, \ldots]$, where $(a_i)_{i \geq 1}$ is a sequence beginning with arbitrarily long palindromes, then $\alpha$ is either transcendental or quadratic.
\end{Theorem}

These results rely on the fact that, in characteristic zero, analogues of Roth’s theorem and of the Subspace Theorem are available for power series fields (see, for instance, \cite{Las00}). However, in order to obtain our results on quasi-periodic continued fractions, a quantitative version of Roth’s theorem over power series fields is required. To the best of our knowledge, no such quantitative version is currently available in a form suitable for our purposes. 
Therefore, another of our results is summarized in the following theorem.

\begin{Theorem}\label{Thm: quantUchi}
Let $K$ be a field of characteristic zero. Let $\alpha \in K((T^{-1}))$ be algebraic of degree $d\geq 2$. Then, for every $0<\varepsilon \le \frac{1}{3}$, the number of solutions of 
\begin{equation}\label{eq: inequality}
 \left | \alpha - \frac{p}{q} \right | < \frac{1}{|q|^{2 + \varepsilon}}
\end{equation}
with $p,q\in K[T]$, $\gcd(p,q)=1$, is less than $\exp(Cd^2\varepsilon^{-2})$ for a positive constant $C$ depending only on $\alpha$.
\end{Theorem}

\begin{remark}
We consider solutions of \eqref{eq: inequality} modulo simultaneous multiplication by elements of $K^{\times}$.
Indeed, if $K$ is infinite and $(p,q)$ is a solution, then so are $(cp,cq)$ for all $c \in K^{\times}$.
\end{remark}

As a consequence, we also obtain an analogue of a result of Davenport and Roth \cite{DR55}. More precisely, let $\alpha \in K((T^{-1}))$ be algebraic and $q_i$ be the denominator of the $i$-th convergent in its continued fraction expansion.
Then,
\[\log \log |q_i| < \frac{{C}i}{\sqrt{\log i}},\]
where $C$ is a positive constant depending only on $\alpha$.

The paper is structured as follows.
In Section~\ref{subsec: diophappr}, we review the fundamental results on Diophantine approximation over function fields, both in positive characteristic and in characteristic zero. These results will be used throughout the proofs of the main theorems. In Section~\ref{subsec: continued}, we recall the definition and the main properties of continued fractions over function fields, again considering both the positive characteristic and the characteristic zero settings.
In Section~\ref{sec: quantit}, we prove Theorem~\ref{Thm: quantUchi}, which provides a quantitative version of Uchiyama's theorem by establishing an explicit upper bound for the number of solutions to \eqref{eq: inequality}. 
Section~\ref{sec: proofs} contains the proofs of the transcendence results for the two classes of quasi-periodic continued fractions introduced above, and for palindromic continued fractions. More precisely, in Section~\ref{subsec: quasiperiod} we prove Theorem~\ref{Thm: mainquasiperiodic} and in Section~\ref{subsec: pali} we prove Theorem~\ref{Thm: pal}. 
Finally, in Section~\ref{sec: spunti}, we discuss promising directions for future research that may lead to improvements in our results.

\section{Preliminaries}\label{sec: prelim}
Given a field $K$, let $K[T]$ be the ring of polynomials with coefficients in $K$, $K(T)$ its field of fractions and $K((T^{-1}))$ the completion of $K(T)$ with respect to the ultrametric norm
\[ \left| \cfrac{a}{b} \right| = \mu^{\deg a - \deg b}, \quad |0|=0 \]
for all non-zero $a/b \in K(T)$ and $\mu$ a fixed real number greater than 1.

\subsection{Diophantine approximation}\label{subsec: diophappr}
In this section, we recall the fundamental results on Diophantine approximation over fields of power series. For further details, see, e.g., \cite{Las00}.
We begin by recalling an analogue of Liouville's theorem, proved by Mahler \cite{Mahler}, which holds for a field $K$ of arbitrary characteristic.

\begin{Theorem}[Mahler \cite{Mahler}]\label{Thm: Mahler}
Let $\alpha \in K((T^{-1}))$ be algebraic of degree $d \ge 2$. There exists a positive real constant $C$, depending only on $\alpha$, such that
\[ \left|\alpha - \cfrac{p}{q} \right| \geq \cfrac{C}{|q|^d} \]
for all $p, q \in K[T]$ with $q \not= 0$.
\end{Theorem}

There is also an analogue of Roth's theorem proved by Uchiyama \cite{Uchi} but only for fields $K$ of characteristic zero.

\begin{Theorem}[Uchiyama \cite{Uchi}]\label{Thm: Uchiyama}
Let $K$ be a field of characteristic zero, and let $\alpha \in K((T^{-1}))$ be algebraic over $K(T)$. Then, for each $\varepsilon > 0$, the inequality
\[ \left| \alpha - \cfrac{p}{q} \right| < \cfrac{1}{|q|^{2 + \varepsilon}}\]
has only finitely many solutions $(p,q)\in K[T]^2$ with $q \not= 0$ and $\gcd(p, q) = 1$.
\end{Theorem}

Finally, Ratliff \cite{Rat} obtained a version of the Subspace Theorem for function fields.

\begin{Theorem}[Ratliff \cite{Rat}]\label{Thm: Ratliff}
Let $K$ be a field of characteristic zero and let $L_1,\ldots,L_n$ be linear forms in $n$ variables with coefficients in $K((T^{-1}))$ that are algebraic over $K(T)$. Suppose that $L_1,\ldots,L_n$ are linearly independent over $K((T^{-1}))$. Then, for all $\varepsilon>0$ there exists a constant $C\in\mathbb R$ and finitely many proper linear subspaces of $K(T)^n$, such that all solutions $(x_1,\ldots,x_n)\in K[T]^n$ of
\begin{equation*}
\prod\limits_{i=1}^n|L_i(x_1,\ldots,x_n)|<\frac{1}{\max\{|x_1|,\ldots,|x_n|\}^{\varepsilon}},
\end{equation*}
either satisfy $\max\{|x_1|,\ldots,|x_n|\}<C$ or they lie in one of the finite number of proper subspaces.
\end{Theorem}

\subsection{Continued fractions}\label{subsec: continued}
In this section, we present the construction of continued fractions over function fields and recall the properties that will be exploited in the following. For more details see, e.g., \cite{Las20, Mala, Sch00}.
Given $\alpha = \sum\limits_{i \leq k} b_i T^i \in K((T^{-1}))$, with $k\in\mathbb Z$, $b_i \in K$ and $b_k \ne 0$, we define the following floor function:
\[ \begin{cases}  \lfloor \alpha \rfloor = \sum\limits_{i=0}^k b_i T^i \in K[T] \quad &\text{if $k \geq 0$}, \\ \lfloor \alpha \rfloor = 0 \quad &\text{otherwise}. \end{cases} \]
Then, the continued fraction expansion $\alpha_0 = [a_0, a_1, \ldots]$ is obtained by the usual recursive algorithm
\[ \begin{cases} a_i = \lfloor \alpha_i \rfloor \\ \alpha_{i+1} = \cfrac{1}{\alpha_i - a_i} \end{cases} \]
for $i = 0, 1, \ldots$.
We define the convergents of a continued fraction as 
\[\frac{p_i}{q_i} := [a_0, \ldots, a_i] \]
for all $i = 0, 1, \ldots$, where the sequences $(p_i)_{i \geq 0}$ and $(q_i)_{i \geq 0}$ satisfy the following recurrences:
\begin{equation*}
\begin{cases}
p_{-2}=0, \quad p_{-1}=1\\
p_{i}=a_{i} p_{i-1}+p_{i-2} 
\end{cases}
, \quad
\begin{cases}
q_{-2}=1, \quad q_{-1}=0 \\
q_{i}=a_{i} q_{i-1}+q_{i-2}
\end{cases}
\end{equation*}
for all $i \geq 0$.
Moreover, the classical identities for continued fractions hold:
\[ \alpha_0 = \cfrac{\alpha_{i+1} p_i + p_{i-1}}{\alpha_{i+1} q_i + q_{i-1}}, \quad p_{i} q_{i-1} - q_{i} p_{i-1} = (-1)^i \]
for all $i \geq 0$. We summarize in the following propositions some classical and useful properties.

\begin{Proposition}\label{prop: convprop}
Given $\alpha_0 = [a_0, a_1, \ldots]\in K((T^{-1}))$, let $(p_i)_{i \geq 0}$ and $(q_i)_{i \geq 0}$ be the sequences of numerators and denominators of convergents, respectively. The following properties hold
\begin{enumerate}
\item $|\alpha_i| = |a_i|$ for all $i \geq 1$.
\item $\gcd(p_i, q_i) = 1$ for all $i \geq 1$.
\item $|p_i| = |a_0a_1+1||a_2| \cdots |a_i|$ for all $i \geq 2$ and $|q_i| = |a_1|\cdots|a_i|$ for all $i \geq 1$.
\item $\left| \alpha - \cfrac{p_i}{q_i} \right| = \cfrac{1}{|q_{i+1}q_i|} < \cfrac{1}{|q_i|^2}$ for all $i \geq 0$.
\end{enumerate}
\end{Proposition}

\begin{Proposition}\label{Pro: difference}
Let $\alpha=[a_0,a_1,\ldots]$ and
$\beta=[b_0,b_1,\ldots]$ be two continued fractions.
If $a_i=b_i$ for all $0\leq i\leq n$, then
\[
|\alpha-\beta|\leq \frac{1}{|q_n|^2}.
\]
\end{Proposition}
Throughout the paper, for any $\alpha \in K((T^{-1}))$ algebraic over $K(T)$, we denote its height by $H(\alpha)$.

\begin{Lemma}[{\cite[Lemma~1]{HMT}}]
\label{Lem: HeiHbaib}
Let $\alpha \in K((T^{-1}))$ be algebraic of degree $d$ whose continued fraction expansion is $\alpha=[a_0, a_1, \ldots, a_{t-1}, \alpha_t]$ where $a_0, \ldots, a_{t-1} \in K[T]$, $\alpha_t \in K((T^{-1}))$. If $|\alpha| \ge 1$ and $|\alpha_t|>1$, then $\alpha_t$ is algebraic of degree $d$ and
\[H(\alpha_t) \le H(\alpha) \left |\prod_{i=0}^{t-1} a_i \right |^{d-2}.\]
\end{Lemma}
Given $\alpha \in K((T^{-1}))$ algebraic with minimal polynomial $X^m + \frac{b_{m-1}}{\bar{b}_{m-1}} X^{m-1} + \frac{b_{m-2}}{\bar{b}_{m-2}} X^{m-2} + \cdots + \frac{b_0}{\bar{b}_{0}}$ with $b_0, \ldots, b_{m-1}, \bar{b}_0, \ldots, \bar{b}_{m-1} \in K[T]$ and $\gcd(\bar{b}_{i}, b_i) =1$ for $i=0, \ldots,m-1$, we define $\sigma(\alpha) = \gcd(\bar{b}_{0}, \ldots, \bar{b}_{m-1})$.

\begin{Lemma}[{\cite[Lemma~3]{HMT}}]
\label{Lem: fgHeiHbaib}
Let $\alpha, \beta \in K((T^{-1}))$ be algebraic of degrees $d$ and $m$, respectively. If $\beta$ is reduced (i.e., $|\beta|>1$ and all of its conjugates have norm less than $1$) then
\[|\alpha-\beta| \ge \frac{1}{H(\alpha)^{m} |\beta|^{d-2} |\sigma(\beta)|^{\max \{m-1,m(d-m+2)-1 \}}}.\]
\end{Lemma}

\section{A quantitative version of Uchiyama's theorem}\label{sec: quantit}
This section is mainly devoted to proving Theorem~\ref{Thm: quantUchi}. We first prove the result for an algebraic element integral over $K[T]$ (Theorem~\ref{thm: DR}) and then extend it to arbitrary algebraic elements. We conclude by deriving Theorem~\ref{Thm: DRfunfields}, which will be used in the next section to prove Theorem~\ref{Thm: mainquasiperiodic}~a). In the following, we assume $K$ with characteristic zero.

Uchiyama proved Theorem~\ref{Thm: Uchiyama} in a way similar to the proof of Roth's theorem. 
Let $\alpha \in K((T^{-1}))$, $\alpha \ne 0$, be integral over $K[T]$ of degree $d \ge 2$ over $K(T)$, i.e., its minimal polynomial is $f(X)=X^{d}+e_{1}X^{d- 1}+\cdots+e_{d}$ in $K[T][X]$.
First, we introduce the notation. 
For all $\varepsilon > 0$, let $m \in \mathbb N$ and $\delta \in \mathbb R$ such that
\begin{equation}\label{eq: U0}
m > d \sqrt{2m}, \quad \cfrac{2m}{m-d\sqrt{2m}} < 2 + \varepsilon, \quad 0 < \delta < 1
\end{equation}
\begin{equation}\label{eq: U3}
2 \eta + (1 + 2 \delta)d\sqrt{2m} < m
\end{equation}
\begin{equation}\label{eq: U1}
\frac{2 m (1+ \delta) + 2 \delta(1+ \delta)}{m-(1+2\delta)d\sqrt{2m}-2 \eta} < 2+\varepsilon, 
\end{equation}
where \[\eta = 7^m \delta^{(1/2)^m}.\] 
Then, Uchiyama \cite[Section~7]{Uchi} proved that the inequality
\begin{equation}\label{eq: U-inequality}
\left | \alpha - \frac{p}{q} \right | < \frac{1}{|q|^{2+\varepsilon}}
\end{equation}
cannot have $m$ solutions $(s_1, t_1), \ldots, (s_m, t_m) \in K[T]^2$, with $\gcd(s_i, t_i) = 1$ for all $i=1, \ldots, m$, that satisfy

\begin{equation}\label{eq: U4}
\log |t_1| > (1 + m \log H(\alpha))\delta^{-2}
\end{equation}
and
\begin{equation}\label{eq: U5}
\frac{\log |t_{j}|}{\log |t_{j-1}|} > \frac{2}{\delta} \qquad 2 \le j \le m.
\end{equation}
Thanks to this, we can prove a quantitative version of Theorem~\ref{Thm: Uchiyama}, analogous to the quantitative version of Roth’s theorem proved by Davenport and Roth \cite[Theorem~1, Corollary~2]{DR55}.

\begin{remark}
Let $K$ be an infinite field, let $\alpha \in K((T^{-1}))$ be algebraic over $K(T)$ and $\varepsilon>0$. Then, \eqref{eq: U-inequality} may have infinitely many solutions $(p,q)$ if we do not require the condition $\gcd(p,q)=1$. Indeed, suppose
\begin{equation}\label{eq: coprimecond}
\left | \alpha - \frac{p}{q} \right | < \frac{1}{\mu^{2+\varepsilon}|q|^{2+\varepsilon}}.
\end{equation}
Then, if \eqref{eq: coprimecond} has a solution $(p,q)$, then it has the (infinitely many) solutions \[((T- \lambda)p,(T- \lambda)q) \]
for all $\lambda \in K$. Indeed,
\[\left | \alpha - \frac{(T- \lambda)p}{(T-\lambda)q} \right | = \left | \alpha - \frac{p}{q} \right | < \frac{1}{\mu^{2+\varepsilon}|q|^{2+\varepsilon}}=\frac{1}{|(T-\lambda)q|^{2+\varepsilon}}.\]
An example is obtained as follows. Let $\beta \in K((T^{-1}))$ be algebraic, and set \[\alpha = \beta \cdot T^{- \deg \lfloor \beta \rfloor  - 4} + T, \qquad \varepsilon = 1, \qquad p=T, \qquad q=1.\] Then $\alpha$ is algebraic and inequality~\eqref{eq: coprimecond} holds.
\end{remark}

In the next theorem, we first deal with integral elements.

\begin{Theorem}\label{thm: DR}
Let $\alpha \in K((T^{-1}))$ be integral over $K[T]$ of degree $d\ge 2$. Then, for every $0<\varepsilon \le \frac{1}{3}$, the number of solutions of
\begin{equation*}
\left | \alpha - \frac{p}{q} \right | < \frac{1}{ |q|^{2+\varepsilon}},
\end{equation*}
with $\gcd(p,q)=1$, is less than $\exp(Cd^2\varepsilon^{-2})$ for some constant $C$ depending only on $\alpha$.
\end{Theorem}
\begin{proof}
Let $m = \lfloor 100 d^2 \varepsilon^{-2} \rfloor + 1$ and $\delta$ such that $\eta = 1$. We prove that this choice of $m$ and $\delta$ satisfies \eqref{eq: U0}, \eqref{eq: U3}, and \eqref{eq: U1}.
Substituting the above value of $m$ and $\delta = \frac{1}{7^{m 2^m}}$ into equations \eqref{eq: U0}, it is straightforward to verify that they hold.

We now prove that \eqref{eq: U3} holds. Since $0< \delta <1$, we have
\[2+(1+2 \delta)d \sqrt{2m}< 2+3d\sqrt{2m}.\]
We show that if $m \ge 19d^2$, then
\begin{equation}\label{eq: U3dim}
2+3d\sqrt{2m} < m.
\end{equation}
Let $g(m):=m-3d\sqrt{2m}-2$. We see that $g'(m) \ge 0$ for all $m \ge \frac{9}{2}d^2$ and $g(19d^2) = (19 - 3 \sqrt{38})d^2 - 2>0$ for $d \ge 2$. Therefore, \eqref{eq: U3dim} holds for all $m \ge 19d^2$, and so does \eqref{eq: U3}. Since $m = \lfloor 100 d^2 \varepsilon^{-2} \rfloor + 1 \ge 900d^2 \ge 19d^2$, the claim follows.

We now prove that \eqref{eq: U1} holds. When the denominator on the left-hand side of \eqref{eq: U1} is positive, this is equivalent to
\[2m\delta + 2\delta(1+\delta)+2(1+2\delta)d\sqrt{2m} + 4+ \varepsilon (1+2\delta)d\sqrt{2m} + 2\varepsilon < \varepsilon m.\]
If the denominator is negative, then the inequality is immediately satisfied.
Using $\delta = \frac{1}{7^{m 2^m}}<1$, $m\delta \le 1$, $4\delta d \sqrt{2m}<1$, and $m<121 d^2 \varepsilon^{-2}$, we derive the following chain of inequalities
\[\begin{aligned}
&2m\delta + 2\delta(1+\delta)+2(1+2\delta)d\sqrt{2m} + 4+ \varepsilon (1+2\delta)d\sqrt{2m} + 2\varepsilon \\
&< 13 + 2d \sqrt{2m} + \varepsilon (1+2\delta)d \sqrt{2m} \\
&< 13 + 22 \sqrt{2} d^2 \varepsilon^{-1} + 11 \sqrt{2}d^2 + 22 \sqrt{2} \delta d^2 \\
&< 14 + 33 \sqrt{2}d^2 \varepsilon^{-1} < \varepsilon m
\end{aligned}\]
proving \eqref{eq: U1}.
Let $(s_1,t_1), (s_2,t_2), \ldots$ with $|t_1|<|t_2|<|t_3|<\cdots$ solutions of \eqref{eq: inequality}, then for all $i\ge 1$
\[\frac{1}{|t_i t_{i+1}|} \le \left | \frac{s_i}{t_i} - \frac{s_{i+1}}{t_{i+1}}\right | \le \max \left \{ \left | \alpha - \frac{s_i}{t_i} \right |, \left | \alpha - \frac{s_{i+1}}{t_{i+1}}\right | \right \} < \frac{1}{|t_i|^{2+\varepsilon}},\]
which implies
\[ \frac{\log |t_{i+1}|}{\log |t_{i}|}>1+ \varepsilon.\]

Let $k>0$ be the smallest integer such that $(1 + \varepsilon)^{k-2}>(1 + m \log H(\alpha))\delta^{-2}$ (see \eqref{eq: U4}). Then, 
\[\log |t_k| > (1 + \varepsilon)^{k-2} \log |t_2| >(1 + \varepsilon)^{k-2}> (1 + m \log H(\alpha))\delta^{-2},\]
using the fact that $\log |t_2| > 1$.

Let $\ell$ be the least integer such that $(1 + \varepsilon)^{\ell}>2 \delta^{-1}$, then
\[\frac{\log |t_{i + \ell}|}{\log |t_{i}|} > \frac{2}{\delta}\]
for all $i \geq 1$. The number of solutions of \eqref{eq: inequality} is less than $k + (m-1)\ell$, otherwise
\[(s_{k}, t_{k}), \ (s_{k+\ell},t_{k+\ell}),  \ldots, \ (s_{k+(m-1)\ell},t_{k+(m-1)\ell})\]
are $m$ solutions of \eqref{eq: inequality} satisfying \eqref{eq: U4} and \eqref{eq: U5}, contradicting Uchiyama's result.

The remainder of the proof is devoted to deriving an upper bound for $k + (m-1)\ell$. The following hold
\[\log (\delta^{-1}) = m 2^m \log 7,\quad \log (1 + \varepsilon) \ge \frac{\varepsilon}{2}, \quad  \ell \log(1 + \varepsilon) > \log (2\delta^{-1}),\]
hence
\begin{equation}\label{eq: leastl}
\ell \ge \frac{\log(2 \delta^{-1})}{\log (1 + \varepsilon)}.\end{equation}
Since $\ell$ is the least integer satisfying \eqref{eq: leastl}, we have \[\ell \le \frac{\log (2 \delta^{-1})}{\log(1+\varepsilon)} + 1 \le \frac{\log 2 + m 2^m \log 7}{\varepsilon/2}+1,\] and so there exists $c_1>0$ such that \[\ell \le c_1 \frac{m 2^m}{\varepsilon}.\]
Moreover, $(1 + \varepsilon)^{k-2}>(1 + m \log H(\alpha))\delta^{-2}$, that is 
\[(k-2) \log(1+\varepsilon)>\log (1 + m \log H(\alpha)) + 2 \log(\delta^{-1}),\] and so
\[k > \frac{\log (1 + m \log H(\alpha)) + 2 \log(\delta^{-1})}{\log(1+\varepsilon)} + 2.\]
Since $k$ is the least integer satisfying the above inequality, we have
\[k \le \frac{\log (1 + m \log H(\alpha)) + 2 \log(\delta^{-1})}{\log(1+\varepsilon)} + 3 \le \frac{\log (1 + m \log H(\alpha)) + m 2^{m+1} \log 7 }{\varepsilon/2} + 3.\]
Therefore, there exists $c_2 > 0$ such that
\[k \le c_2 \frac{m 2^m}{\varepsilon}.\]
Combining these results, we obtain
\[k + (m-1) \ell \le c_2 \frac{m 2^m}{\varepsilon} + (m-1) \left ( c_1 \frac{m2^m}{\varepsilon} \right ) \le c_3\frac{m^2 2^m}{\varepsilon},\]
for a suitable $c_3>0$.
Since $m = \lfloor 100 d^2 \varepsilon^{-2}\rfloor + 1$, there exists $c_4 >0$ such that
\[k + (m-1)\ell \le \exp(c_4 d^2 \varepsilon^{-2}).\]\end{proof}

The remainder of this section is devoted to generalizing the results of Theorem~\ref{thm: DR} to the case where $\alpha$ is algebraic over $K(T)$ (not necessarily integral). The following proposition will be used to prove this result.

\begin{Proposition}\label{prop: ineqAPQ}
Let $h, s_1, t_1, s_2, t_2 \in K[T]$, $\alpha \in K((T^{-1}))$, and $\varepsilon>0$. 
Suppose that $s_1/t_1 \not= s_2/t_2$ and satisfy
\[\left | \alpha - \frac{h s_i}{t_i} \right | < \frac{|h|}{|t_i|^{2+ \varepsilon}}\]
with $\gcd(s_i,t_i)=1$ and $i=1,2$. Then $|t_1| \neq |t_2|$.
\end{Proposition}
\begin{proof}
Let us suppose $|t_1|=|t_2|$. Then, we have
\[\frac{|h|}{|t_1|^2} \le \left | \frac{hs_1}{t_1} - \frac{hs_2}{t_2} \right | = \left | \left ( \alpha - \frac{hs_1}{t_1} \right ) - \left ( \alpha - \frac{hs_2}{t_2} \right ) \right | < \frac{|h|}{|t_1|^{2 + \varepsilon}},\]
which implies
\[1 < \frac{1}{|t_1|^{\varepsilon}},\]
which is impossible.
\end{proof}

\begin{proof}[Proof of Theorem~\ref{Thm: quantUchi}]
Let $g(X)=e_dX^d + \cdots +e_1X + e_0\in K[T][X]$ be a polynomial of the smallest possible degree for which $\alpha$ is a root, and set $\beta=e_d\alpha$. Then $\beta$ is integral over $K[T]$. Indeed,
\[
\bar{g}(X)=X^d+\sum_{i=0}^{d-1}e_d^{d-1-i}e_i X^i\]
is a monic polynomial in $K[T][X]$, and
\[
\bar{g}(\beta)=\bar{g}(e_d\alpha)=e_d^{d-1}g(\alpha)=0.
\]
We denote $e_d \in K[T]$ by $a$. Since
\[
\left|\alpha-\frac{p}{q}\right|
<\frac{1}{|q|^{2+\varepsilon}}
\iff
\left|\beta-\frac{ap}{q}\right|
<\frac{|a|}{|q|^{2+\varepsilon}},
\]
it suffices to show that the inequality
\begin{equation}\label{eq: APQ}
\left|\beta-\frac{ap}{q}\right|
<
\frac{|a|}{|q|^{2+\varepsilon}},
\quad \gcd(p,q)=1,
\end{equation}
admits only finitely many solutions, and to obtain an upper bound for their number.

To do so, we show that, for all $b \mid a$, $b \in K[T]$, $b$ monic, we can bound the number of solutions of \eqref{eq: APQ} such that $b=\gcd(a,q)$. Let us fix $b \mid a$ monic and let us suppose that there exist infinitely many solutions of \eqref{eq: APQ} $\{ (s_i, t_i) \}_{i \ge 0}$ such that $b=\gcd(a,t_i)$ for all $i\ge 0$. By Proposition~\ref{prop: ineqAPQ}, $|t_i| \ne |t_j|$ for all $i \ne j$ and, without loss of generality, we can assume $|t_i| < |t_{i+1}|$ for all $i \ge 0$. Setting $a=\tilde{a}b$ and $t_i=\tilde{t}_ib$, with $\gcd(\tilde{a},\tilde{t}_i)=1$, we have
\[\left | \beta - \frac{\tilde{a}s_i}{\tilde{t}_i} \right | < \frac{|\tilde{a}|}{|\tilde{t}_i|^{2+ \varepsilon}|b|^{1+\varepsilon}}\]
for all $i \geq 0$. Since $|\tilde{t}_i|<|\tilde{t}_{i+1}|$ for all $i\ge 0$, there exists $k\ge 0$ such that
\[\frac{|\tilde{a}|}{|b|^{1+\varepsilon}}< |\tilde{t}_i|^{\varepsilon/2}\]
for all $i \geq k$ and so
\begin{equation}\label{eq: epsmezzi}\left | \beta - \frac{\tilde{a} s_i}{\tilde{t}_i} \right | < \frac{1}{|\tilde{t}_i|^{2+\varepsilon/2}}\end{equation}
for all $i \geq k$. 
Since $\beta$ is integral over $K[T]$, by Theorem~\ref{thm: DR}, the inequality~\eqref{eq: epsmezzi} cannot have infinitely many solutions $\gcd(\tilde{a}s_i, \tilde{t}_i)=1$. In particular, by Theorem~\ref{thm: DR}, the number of solutions of \[\left | \beta - \frac{p}{q}\right | < \frac{1}{|q|^{2+\varepsilon/2}}, \quad \gcd(p,q) = 1\] is at most $\exp(4Cd^2 \varepsilon^{-2})$. 
Moreover, using Proposition~\ref{prop: ineqAPQ}, the number of solutions of \eqref{eq: APQ} such that $\gcd(a,q)=b$ and $\left | \frac{q}{b} \right |\le \frac{|\tilde{a}|^{2/\varepsilon}}{|b|^{2/\varepsilon+2}}$ is at most
\[\frac{2}{\varepsilon} \deg(\tilde{a}) - \left ( 1 + \frac{2}{\varepsilon} \right )\deg(b).\]
Therefore, the number of solutions of \eqref{eq: APQ} such that $b=(a,t_i)$ for all $i \ge 0$ is less than
\[\exp(4Cd^2 \varepsilon^{-2}) + \frac{2}{\varepsilon} \deg(\tilde{a}) - \left ( 1 + \frac{2}{\varepsilon} \right )\deg(b).\]
The number of monic divisors of $a$ is at most $2^{\deg(a)}\le 2^{\log(H)}$, where 
\[H := \max \{|e_0|, \ldots, |e_d| \}.\]
Therefore, the number of solutions of \eqref{eq: inequality} is at most
\[2^{\log(H)} \left ( \exp(4Cd^2 \varepsilon^{-2}) + \frac{2}{\varepsilon} \deg(a) \right ) \le \exp(\bar{C} d^2 \varepsilon^{-2}),\]
for a suitable $\bar{C}>0$ depending only on $\alpha$.
\end{proof}

\begin{Theorem}\label{Thm: DRfunfields}Let $\alpha \in K((T^{-1}))$ be algebraic of degree $d\ge 2$. Then,
\[\log \log |q_i| < \frac{C i}{\sqrt{\log i}}\]
for all $i \geq 2$, for some constant $C$ depending only on $\alpha$, where $q_i$ is the denominator of the $i$-th convergent of $\alpha$.
\end{Theorem}
\begin{proof}
Let us define 
\[\gamma_i := \frac{\deg q_{i+1}}{\deg q_i}\]
for all $i \geq 1$ and so $|q_{i+1}| = |q_{i}|^{\gamma_i}$.
By Proposition~\ref{prop: convprop},
\[\left | \alpha - \frac{p_i}{q_i} \right | < \frac{1}{|q_iq_{i+1}|} = \frac{1}{|q_i|^{\gamma_i + 1}} \quad \text{for all } i \ge 1,\]
that is
\[\deg \left ( \alpha - \frac{p_i}{q_i} \right ) < - \deg q_i - \deg q_{i+1} = - (\gamma_i + 1)\deg q_i.\]
If $\alpha$ is algebraic of degree $d$, then by Theorem~\ref{Thm: Mahler} we have
\[\frac{c_1}{|q_i|^d} \le \left | \alpha - \frac{p_i}{q_i} \right | = \frac{1}{|q_i|^{\gamma_i + 1}}\]
for all $i \geq 0$, where $c_1>0$ is a constant depending only on $\alpha$. This implies that $(\gamma_i + 1)\deg q_i \le d \deg q_i - \log c_1$ from which 
\[\gamma_i + 1 \le d - \frac{\log c_1}{\deg q_i},\]
and so $\gamma_i +1 \le d + c_2$, where $c_2$ is a positive constant.
We have \[\deg q_i = \deg q_1 \cdot \frac{\deg q_2}{\deg q_1} \cdots \frac{\deg q_i}{\deg q_{i-1}},\] from which
\[\log (\deg q_i) = \log (\deg q_1) + \sum_{j=1}^{i-1} \log\left ( \frac{\deg q_{j+1}}{\deg q_j} \right ) = \log (\deg q_1) + \sum_{j=1}^{i-1} \log \gamma_j.\]
For all $0<\varepsilon<\frac{1}{3}$ we define
\[S_1 := \{ 1 \le j \le i-1 \mid \gamma_j \le 1+\varepsilon\}, \quad S_2 := \{ 1 \le j \le i-1 \mid \gamma_j > 1+\varepsilon\}\]
Therefore
\[\sum_{j=1}^{i-1} \log \gamma_j = \sum_{j \in S_1} \log \gamma_j + \sum_{j \in S_2} \log \gamma_j.\]
We have that \[\sum_{j \in S_1} \log \gamma_j \le \sum_{j \in S_1} \log(1+ \varepsilon) \le \sum_{j \in S_1} \varepsilon \le i \varepsilon.\]
If $\gamma_j > 1+\varepsilon$, then $|q_{j+1}|>|q_j|^{1+\varepsilon}$ and $(p_j, q_j)$ is a solution of \eqref{eq: inequality}.
Since $\deg q_j \ge j$ and the number of solutions of \eqref{eq: inequality} is less than $\exp(c_3 \varepsilon^{-2})$, we have $|S_2| < \exp(c_3 \varepsilon^{-2})$, and hence
\[\sum_{j \in S_2} \log \gamma_j < \exp(c_3 \varepsilon^{-2}) \cdot \log(d + c_2).\]
In conclusion,
\[\log \log |q_i| = \log \deg q_i < \log \deg q_1 + i \varepsilon + \exp(c_3 \varepsilon^{-2})c_4 <i \varepsilon + c_5\exp(c_3 \varepsilon^{-2}).
\]
This is the same conclusion as the one obtained by Davenport and Roth in \cite{DR55}. 
We choose $c_6 \in \mathbb R$ such that $c_6^2 > c_3$ and, for all sufficiently large $i$, set $\varepsilon = \frac{c_6}{\sqrt{\log i}}$. With this choice, we obtain
\[\begin{aligned}\log \log |q_i| <c_6 \frac{i}{\sqrt{\log i}} + c_5 i^{c_3/c_6^2} \ll \frac{i}{\sqrt{\log i}},
\end{aligned}\]
which completes the proof.
\end{proof}

\section{Proofs of the main results}\label{sec: proofs}
\subsection{Quasi-periodic continued fractions}\label{subsec: quasiperiod}
This section is devoted to proving the transcendence criteria for the class of quasi-periodic continued fractions in Theorem~\ref{Thm: mainquasiperiodic}. First, we prove part~a), where we assume that $r_i<Cn_i$ for some constant $C$, while no assumption on the absolute values of the partial quotients is required. The proof follows a similar argument to that of \cite[Theorem~3.1]{ADH2}. The difference lies in the final tool that we use to derive a contradiction. Instead of \cite[Lemma~3.3]{ADH2}, we use the finer estimate proved in Theorem~\ref{Thm: quantUchi}, which holds only over fields of characteristic zero.
Finally, we prove part~b) of Theorem~\ref{Thm: mainquasiperiodic}, where we assume that the sequence $(|q_i|^{1/i})_{i \geq 1}$ is bounded, instead of requiring the boundedness of the partial quotients. To establish this result, we exploit the Subspace Theorem~\ref{Thm: Ratliff} in a novel way that also differs from the approach used to prove the analogous result over the real numbers.

\begin{proof}[Proof of Theorem~\ref{Thm: mainquasiperiodic}~a)]

Suppose, for contradiction, that $\alpha=[a_0,a_1,\ldots]$ is algebraic of degree $d>2$. Let us denote, for all $i\geq 1$, $B_{n_i}=[a_{n_i},\ldots,a_{n_i+r_i-1}]$, and
\[\beta^{(i)}=[\overline{B_{n_i}}]=[B_{n_i},B_{n_i},\ldots].\]
Since $\beta^{(i)}$ has a purely periodic expansion, it is a reduced quadratic Laurent series, as defined in Lemma~\ref{Lem: fgHeiHbaib} (see \cite[Lemma~2.1.10]{Mala}).
Let us denote, for all $j\geq 0$,
\[\frac{p'_j}{q'_j}=[a_{n_i},\ldots,a_{n_i+j}],\]
the convergents of the continued fraction of $\beta^{(i)}$.
Then, 
\[\beta^{(i)}=[\overline{a_{n_i},\ldots,a_{n_i+r_i-1}}]=\frac{\beta^{(i)}p'_{r_i-1}+p'_{r_i-2}}{\beta^{(i)}q'_{r_i-1}+q'_{r_i-2}},\]
as the length of the block $B_{n_i}$ is equal to $r_i$. This means that
\[q'_{r_i-1}(\beta^{(i)})^2+(q'_{r_i-2}-p'_{r_i-1})\beta^{(i)}-p'_{r_i-2}=0.\]
Without loss of generality, we can assume $|\alpha|\ge 1$ and by Lemma~\ref{Lem: HeiHbaib} the $n_i$-th complete quotient of $\alpha$ satisfies
\begin{equation}\label{Eq: Hfni}
H(\alpha_{n_i}) \le H(\alpha) \left|\prod\limits_{j=0}^{n_i-1} a_j\right|^{d-2},
\end{equation}
as it is algebraic of degree $d$. Then, since $\beta^{(i)}$ and $\alpha_{n_i}$ share the first $\lambda_ir_i$ partial quotients, by Proposition~\ref{Pro: difference},
\begin{equation}\label{Eq: fgupperbound}|\alpha_{n_i}-\beta^{(i)}|\leq \frac{1}{|q'_{\lambda_ir_i-1}|^2}=\frac{1}{\prod\limits_{j=n_i+1}^{n_i+\lambda_ir_i-1} |a_j|^2}= |a_{n_i}|^2\left(\prod\limits_{j=n_i}^{n_i+r_i-1} |a_j|\right)^{-2\lambda_i}.\end{equation}

Moreover, by Lemma~\ref{Lem: fgHeiHbaib} and \eqref{Eq: Hfni},
\begin{equation}
\begin{split}
\label{Eq: fglowerbound}
|\alpha_{n_i}-\beta^{(i)}|&\geq H(\alpha_{n_i})^{-2}|\beta^{(i)}|^{2-d}|\sigma(\beta^{(i)})|^{-2d+1}\geq\\
&\geq H(\alpha)^{-2}|a_{n_i}|^{2-d}|q'_{r_i-1}|^{-2d+1}\left|\prod\limits_{j=0}^{n_i-1} a_j\right|^{-2d+4}.
\end{split}
\end{equation}
By combining the bounds in \eqref{Eq: fgupperbound} and \eqref{Eq: fglowerbound}, we obtain

\begin{equation*}|a_{n_i}|^{-2}\left(\prod\limits_{j=n_i}^{n_i+r_i-1} |a_j|\right)^{2\lambda_i}\leq H(\alpha)^2|a_{n_i}|^{d-2}|q'_{r_i-1}|^{2d-1}\left|\prod\limits_{j=0}^{n_i-1} a_j \right|^{2d-4}.\end{equation*}
Notice that $|q'_{r_i-1}|=\prod\limits_{j=n_i+1}^{n_i+r_i-1}|a_j|$, as it is the denominator of $[B_{n_i}]=[a_{n_i},\ldots,a_{n_i+r_i-1}]$. Therefore, we obtain
\begin{equation*}
2\lambda_i\sum\limits_{j=n_i}^{n_i+r_i-1} \deg a_j\leq 2\log H(\alpha) + d\deg a_{n_i} +(2d-4)\sum\limits_{j=0}^{n_i-1}\deg a_j+(2d-1)\sum\limits_{j=n_i+1}^{n_i+r_i-1}\deg a_j.
\end{equation*}
Since $\deg a_j \geq 1$ for all $j$, $2d-4<2d-1$ and $d<2d-1$, we obtain
\[2\lambda_i r_i \leq 2\lambda_i\sum\limits_{j=n_i}^{n_i+r_i-1} \deg a_j \leq 2\log H(\alpha) +(2d-1)\sum\limits_{j=0}^{n_i+r_i-1}\deg a_j .\]
For sufficiently large $i$, the main term is $\sum\limits_{j=0}^{n_i+r_i-1}\deg a_j$, so that
\[\lambda_i\ll r_i\lambda_i\ll \sum\limits_{j=0}^{n_i+r_i-1}\deg a_j.\]
By Theorem~\ref{Thm: DRfunfields} and the latter inequalities, we get
\[\log \lambda_i \ll \log\left(\sum\limits_{j=0}^{n_i+r_i-1}\deg a_j \right) \ll \frac{n_i+r_i-1}{\sqrt{\log (n_i+r_i-1)}}\ll\frac{n_i}{\sqrt{\log n_i}}.\]
This means that, for all $i$,
\[\frac{\log \lambda_i \sqrt{\log n_i}}{n_i}< C,\]
for some constant $C$ independent of $i$, and this contradicts the assumption \eqref{Eq: HPquasiperA}. Therefore, $\alpha$ is transcendental.
\end{proof}

To prove part~b) of Theorem~\ref{Thm: mainquasiperiodic}, we first need the following Lemma. 

\begin{Lemma} \label{pro: poly}
Let $\xi$ be a quasi-periodic continued fraction as defined above. Suppose that $|q_i|\le M^i$ for all $i \ge 0$, for a constant $M>1$.
For all $i \ge 0$, let
\[\alpha^{(i)} = [a_0, \ldots, a_{n_i -1}, \overline{a_{n_i}, \ldots, a_{n_i + r_i -1}}].\]
Then, each $\alpha^{(i)}$ is a quadratic irrational and a root of a polynomial \[f_i X^2 + g_i X + h_i\] such that
\[\max\{|f_i|, |g_i|, |h_i|\} < M^{2 n_i + r_i}. 
\]
Moreover, there exists a sequence $l_i$ of non-negative integers such that, by defining
\begin{equation*}
    u_i = T^{l_i}f_i, \qquad
    v_i = T^{l_i}g_i, \qquad 
    w_i = T^{l_i}h_i
\end{equation*}
we have
\[ \max \{|u_i|,|v_i|,|w_i|\} < M^{2n_i+r_i} \]
and
\[ \lim_{i \to +\infty} \max \{|u_i|,|v_i|,|w_i| \} = +\infty. \]
\end{Lemma}

\begin{proof}
The quadratic irrationals $\alpha^{(i)}$ are  roots of the polynomials $f_i X^2 + g_i X + h_i$, with
\[\begin{aligned}
f_i &= q_{n_i+r_i-1}q_{n_i-2} - q_{n_i+r_i-2}q_{n_i-1}, \\
g_i &= p_{n_i-1}q_{n_i+r_i-2} - p_{n_i-2}q_{n_i+r_i-1} - p_{n_i+r_i-1}q_{n_i-2} + p_{n_i+r_i-2}q_{n_i-1} , \\
h_i &= - p_{n_i+r_i-1}p_{n_i-2} + p_{n_i+r_i-2}p_{n_i-1}.  
\end{aligned}\]
The boundedness of the sequence $(|q_i|^{1/i})_{i\geq1}$ implies that of $(|p_i|^{1/i})_{i\geq1}$. Therefore, enlarging $M$ if necessary, we may assume that \(
|p_i|<M^i\) and \(|q_i|<M^i\) for every $i\geq1$.
Then, the following inequalities hold:
\[
|f_i| \le M^{2 n_i + r_i-3}< M^{2 n_i + r_i}, \ \ |g_i| \le M^{2 n_i + r_i-3}< M^{2 n_i + r_i}, \ \
|h_i| \le M^{2 n_i + r_i-3}< M^{2 n_i + r_i}.  
\]
Let $H_i := \max \{|f_i|,|g_i|,|h_i| \}$. Define
\[ l_i := \max \{ 0,  (2n_i +r_i) \lfloor \log M \rfloor - \lceil \log H_i \rceil -1 \}.\]
Then, $l_i \in \mathbb{Z}$ and $l_i \ge 0$ for all $i \ge 0$ and $u_i, v_i, w_i \in K[T]$.
Without loss of generality, we assume that $M>\mu$. By construction, for all $i$ we have
\[ \frac{M^{2 n_i + r_i}}{\mu^{2 n_i + r_i + 2} \cdot H_i} = \mu^{(2n_i +r_i) ( \log M -1) - ( \log H_i + 1) -1} \le |T^{l_i}| \le \frac{M^{2 n_i + r_i}}{\mu \cdot H_i}.\]
We have
\[\max \{|u_i|, |v_i|, |w_i| \} = |T^{l_i}| \max \{|f_i|, |g_i|, |h_i| \}\]
and for all $i$
\[ \frac{M^{2 n_i + r_i}}{\mu^{2 n_i + r_i + 2}} \le |T^{l_i}| \max \{|f_i|, |g_i|, |h_i| \} \le \frac{M^{2 n_i + r_i}}{\mu} < M^{2 n_i + r_i}.\]
Since $\frac{M^{2 n_i + r_i}}{\mu^{2 n_i + r_i + 2}}\rightarrow +\infty$, this completes the proof.
\end{proof}

\begin{proof}[Proof of Theorem~\ref{Thm: mainquasiperiodic}~b)]

Let us define 
\[\alpha^{(i)} := [a_0, \ldots, a_{n_i -1}, \overline{a_{n_i}, \ldots, a_{n_i + r_i -1}}] \quad \text{for all } i \ge 0.\]
By Lemma~\ref{pro: poly}, $\alpha^{(i)}$ are quadratic irrationals, roots of \[u_i X^2 + v_i X + w_i,\]
where $u_i,v_i,w_i \in K[T]$ satisfy the bounds stated in the above lemma.
Since $\alpha$ and $\alpha^{(i)}$ have the same initial $n_i + \lambda_ir_i$ partial quotients, by Proposition~\ref{Pro: difference}, we have
\begin{equation}\label{eq: ineqNk}
\left | \alpha - \alpha^{(i)} \right | \le \frac{1}{|q_{N_i}|^2},
\end{equation}
where $N_i = n_i + \lambda_ir_i -1$. Since $|q_i| \le M^i$, from Lemma~\ref{pro: poly}
\[H_i := \max \{|u_i|,|v_i|,|w_i|\} < M^{2 n_i + r_i}.\]

Let us suppose $\alpha$ algebraic of degree $d \ge 3$.
We now prove that, for all $\varepsilon>0$ there exists a subsequence $\Lambda \subseteq \mathbb N$ such that 
\begin{equation}\label{eq: epsilon}
H_i^{t} < |q_{N_i}|^2
\end{equation}
for all $i \in \Lambda$, where \(t = \max \{ 3+\varepsilon , d+1 \}.\)
Indeed, we have that $|q_{N_i}| = |q_{n_i + \lambda_i r_i -1}| \ge \mu^{n_i + \lambda_i r_i -1}$ and $H_i<M^{2 n_i + r_i}$.
Since $\limsup_{i \rightarrow + \infty} \frac{\lambda_i}{n_i} = + \infty$, for all $\bar{C} \in \mathbb R$, $\bar{C}>0$, there exists $\Gamma \subseteq \mathbb N$ such that, $\lambda_i > \bar{C} n_i$ for all $i \in \Gamma$. Hence, there exists $\Gamma \subseteq \mathbb N$ such that
\[\log  H_i^t = t \log H_i< t ( 2 n_i + r_i ) \log M< \lambda_i + t r_i \log M,\]
for a suitable choice of $\bar{C}$.
Since $\lim\limits_{i \in \Gamma, \, i \rightarrow + \infty} \lambda_i = + \infty$, we have that for $i \in \Gamma$ and $i \gg 0$
\[
\lambda_i + t r_i \log M < \lambda_i + \lambda_i r_i 
\le 2 \lambda_i r_i < 2(n_i + \lambda_i r_i -1)\le \log(|q_{N_i}|^2).\]
Therefore, there exists $\Lambda \subseteq \Gamma$ such that \eqref{eq: epsilon} holds.

We prove that, for all but finitely many $i$, $v_i \ne 0$. Indeed, let us suppose $v_i =0$ for infinitely many $i \in \Omega \subseteq \Lambda$. Then, $u_i (\alpha^{(i)})^2 + w_i=0$ and $u_i \ne 0$ for all $i \in \Omega$. Let us consider the following linear forms
\[P_1(X,Y) = X, \quad P_2(X,Y) = \alpha^2X + Y.\]
The following hold
\[\begin{aligned}
|P_2(u_i,w_i)| &= \left |\alpha^2 u_i + w_i \right | = \left |\alpha^2 u_i + w_i -((\alpha^{(i)})^2u_i + w_i) \right | \\
&= \left | u_i \right | \cdot \left |\alpha^2 - (\alpha^{(i)})^2 \right | \le H_i \cdot \left |\alpha - \alpha^{(i)} \right | \le \frac{H_i}{\left |q_{N_i} \right |^2},
\end{aligned}\]
where the second last inequality follows assuming, without loss of generality, that $a_0=0$. From this, we obtain that for all $i \in \Omega$
\[\left |P_1(u_i,w_i)P_2(u_i,w_i) \right | \max \{ |u_i|,|w_i| \}^{\varepsilon}<1.\]
Since $( H_i )_{i \in \Omega}$ is unbounded, from the Subspace Theorem $(u_i,w_i)$ lie in a finite union of proper linear subspaces of $K(T)^2$ for all $i \in \Omega$. Hence, there exist $\tilde{u}, \tilde{w} \in K(T)$, not all zero, such that
\[\tilde{u}u_i + \tilde{w}w_i =0 \]
for infinitely many $i \in \Omega$. We may assume that the above relation holds for every $i\in\Omega$. This implies that
\[0 = \tilde{u}u_i + \tilde{w}w_i - \tilde{w}(u_i (\alpha^{(i)})^2 + w_i) = u_i(\tilde{u} - \tilde{w}(\alpha^{(i)})^2 )\]
and so
\[\tilde{u} - \tilde{w}(\alpha^{(i)})^2  =0\]
for all $i \in \Omega$.
Since all the elements $\alpha^{(i)}$, with $i\in\Omega$, are roots of the same nonzero polynomial
\[\tilde{w}X^2-\tilde{u}\in K(T)[X],\]
they take at most two distinct values. Therefore, there exists an infinite subset $\Omega'\subseteq\Omega$ such that
\(\alpha^{(i)}=\alpha^{(j)}\)
for all $i,j\in\Omega'$.
This implies that $\alpha = \alpha^{(i)}$ for $i \in \Omega'$, contradicting the assumption that $\alpha$ is not a quadratic irrational. 
Removing these finitely many indices from $\Lambda$, we may
assume that $v_i\neq0$ for every $i\in\Lambda$.

We define the following linear forms:
\[L_1(X,Y,Z) = X, \quad L_2(X,Y,Z) = Y, \quad L_3(X,Y,Z) = \alpha^2 X + \alpha Y + Z.\]
The following holds
\[\begin{aligned}|L_3(u_i,v_i,w_i)| &= \left |u_i(\alpha^2 - (\alpha^{(i)})^2) + v_i(\alpha- \alpha^{(i)}) \right |\\
&\le \max \left \{ \left |u_i (\alpha^2 - (\alpha^{(i)})^2) \right |, \left | v_i(\alpha- \alpha^{(i)}) \right | \right \} \\
&\le \max \{|u_i|,|v_i|\} \cdot \left |\alpha - \alpha^{(i)} \right | \le \frac{H_i}{|q_{N_i}|^2}\end{aligned}\]
where, in the second last inequality we have used the fact that $a_0=0$ and in the last inequality we have used \eqref{eq: ineqNk}.
For all $i \in \Lambda$ we have
\[|L_1(u_i,v_i,w_i) L_2(u_i,v_i,w_i) L_3(u_i,v_i,w_i)| \max \{|u_i|,|v_i|,|w_i|\}^{\varepsilon} < 1 \]
and so, from the Subspace Theorem, $( H_i )_{i \in \Lambda}$ is bounded or $(u_i,v_i,w_i)$
are contained in a finite union of proper subspaces of $K(T)^3$ for all $i \in \Lambda$.
By Lemma~\ref{pro: poly}, 
\[\lim_{i \to + \infty}\max \{|u_i|,|v_i|,|w_i|\} = + \infty.\] 
Hence, there exist $\bar u,\bar v,\bar w\in K(T)$, not all zero, such that, for infinitely many $i$
\[\bar{u}u_i + \bar{v} v_i + \bar{w} w_i =0.\]
Using the fact that $u_i (\alpha^{(i)})^2 + v_i\alpha^{(i)} + w_i =0$ we have
\[u_i(\bar{u}- \bar{w}(\alpha^{(i)})^2) + v_i(\bar{v} - \bar{w} \alpha^{(i)}) =0\]
from which
\[\frac{u_i}{v_i} = - \frac{\bar{v} - \bar{w} \alpha^{(i)}}{\bar{u} - \bar{w} (\alpha^{(i)})^2}.\]

We have that $\bar{w} \ne 0$.
Indeed, if $\bar{w} = 0$ then $v_i/u_i = h \in K(T)$ and
\[(\alpha^{(i)})^2 + h \alpha^{(i)} + \frac{w_i}{u_i}=0.\]
Let $\theta = \alpha^2 + h \alpha$. Since $\alpha$ is not quadratic, $\theta \not\in K(T)$.
The sequences
\[\left | \frac{w_i}{u_i} \right | = \left | (\alpha^{(i)})^2 + h \alpha^{(i)} \right | \quad \text{and} \quad \left | \frac{v_i}{u_i} \right | = |h| \]
are bounded, because $\alpha^{(i)} \to \alpha$. Therefore, $H_i \ll |u_i|$ and so $|u_i| \to +\infty$.
The following holds
\[\begin{aligned}\left | \theta + \frac{w_i}{u_i} \right | &= \left |\alpha^2 + h \alpha - (\alpha^{(i)})^2  - h \alpha^{(i)} \right |= \left |\alpha - \alpha^{(i)} \right |\cdot \left | \alpha + \alpha^{(i)} + h \right | \\
&< \frac{C^{\ast}}{|q_{N_i}|^2} < \frac{C^{\ast}}{H_i^t}< \frac{C^{\ast}}{|u_i|^t} < \frac{C^{\ast}}{|u_i|^{d+1}}, \end{aligned}\]
where we have used \eqref{eq: epsilon} and that $\left | \alpha + \alpha^{(i)} + h \right |<C^{\ast}$ for a suitable positive constant $C^{\ast}$.
By Theorem~\ref{Thm: Mahler}, $\theta$ is transcendental, and so $\alpha$.
Therefore,
\[\lim_{i \rightarrow + \infty} \frac{u_i}{v_i} = - \frac{\bar{v} - \bar{w} \alpha}{\bar{u} - \bar{w} \alpha^2} = \psi.\]
Since we have supposed $\alpha$ not a quadratic irrational, we have that $\psi \not\in K(T)$.
Hence, $\psi$ is algebraic of degree $2 \le d' \le d$. 

The following equality holds
\begin{equation*}
\left | \frac{u_i}{v_i} - \psi \right | = |\alpha - \alpha^{(i)}| \cdot \frac{|\bar{w}^2 \alpha^{(i)} \alpha + \bar{u}\bar{w} -\bar{v}\bar{w} (\alpha + \alpha^{(i)})|}{|\bar{w}(\alpha^{(i)})^2 - \bar{u}| \cdot |\bar{w} \alpha^2 - \bar{u}|}.
\end{equation*}
Since
\[\lim_{i \rightarrow + \infty}  \frac{\left |\bar{w}^2 \alpha^{(i)} \alpha + \bar{u}\bar{w} -\bar{v}\bar{w} (\alpha + \alpha^{(i)}) \right |}{\left |\bar{w} (\alpha^{(i)})^2 - \bar{u} \right | \cdot \left |\bar{w} \alpha^2 - \bar{u}\right |} = \frac{\left |\bar{w}^2 \alpha^2 + \bar{u}\bar{w} -2\bar{v}\bar{w} \alpha \right |}{\left |\bar{w} \alpha^2 - \bar{u}\right | \cdot \left |\bar{w} \alpha^2 - \bar{u}\right |} = \tilde{C},\]
and since the norm is discrete, for $i\gg0$
\[\left | \frac{u_i}{v_i} - \psi \right | = \tilde{C} \left |\alpha - \alpha^{(i)} \right |.\]

We prove that $|v_i| \to + \infty$. The two sequences
\[\left | \frac{u_i}{v_i} \right | \quad \text{and} \quad \left | \frac{w_i}{v_i} \right | = \left | \frac{u_i}{v_i} (\alpha^{(i)})^2 + \alpha^{(i)}\right |\]
are bounded, and so $H_i \ll |v_i|$. Since $H_i \to + \infty$, so is $|v_i|$.
Therefore, 
\begin{equation}\label{eq: finale}
\left | \frac{u_i}{v_i} - \psi \right | < \frac{\tilde{C}}{|q_{N_i}|^2}< \frac{\tilde{C}}{|v_i|^{t}} \le \frac{\tilde{C}}{|v_i|^{d+1}} \le \frac{\tilde{C}}{|v_i|^{d'+1}} \end{equation}
for $i \gg 0$. Equation~\eqref{eq: finale} contradicts Theorem~\ref{Thm: Mahler}, hence $\alpha$ is either transcendental or quadratic.
\end{proof}

\subsection{Palindromic continued fractions}\label{subsec: pali}

In this section we prove the transcendence of palindromic continued fractions over $ K((T^{-1}))$. Our approach is based on a substantially different application of the Subspace Theorem from that employed in the real setting.

We say that $[0, a_1, a_2, \ldots]$ is a palindromic continued fraction if it begins with arbitrarily long palindromes, i.e., if there exist infinitely many integers $n$ such that $(a_1, \ldots, a_n)$ is a palindrome. Consequently, for each such index $n$, we have $p_n = q_{n-1}$, since
\[ \begin{pmatrix}  a_1 & 1 \\ 1 & 0 \end{pmatrix} \cdots \begin{pmatrix}  a_n & 1 \\ 1 & 0 \end{pmatrix} = \begin{pmatrix}  q_n & q_{n-1} \\ p_n & p_{n-1} \end{pmatrix} \]
and the latter matrix is symmetric because $(a_1, \ldots, a_n)$ is a palindrome.

\begin{proof}[Proof of Theorem~\ref{Thm: pal}]

Let $n$ be such that the word $(a_1, \ldots, a_n)$ is a palindrome.
Since
\[\alpha=\frac{\alpha_{n+1}p_n+p_{n-1}}{\alpha_{n+1}q_n+q_{n-1}}, \quad \alpha-\frac{p_n}{q_n}=\frac{(-1)^{n}}{q_n(\alpha_{n+1}q_n+q_{n-1})},\]
we obtain
\[\left|q_n\alpha-p_n\right|=\left|\frac{1}{\alpha_{n+1}q_n+q_{n-1}}\right|=\left|\frac{1}{\alpha_{n+1}q_n}\right|.\]
Similarly, from
\[\alpha^{-1}=\frac{\alpha_{n+1}q_n+q_{n-1}}{\alpha_{n+1}p_n+p_{n-1}}, \quad \alpha^{-1}-\frac{q_n}{p_n}=\frac{(-1)^{n+1}}{p_n(\alpha_{n+1}p_n+p_{n-1})}\]
we get
\[\left|p_n\alpha^{-1}-q_n\right|=\left|\frac{1}{\alpha_{n+1}p_n}\right|.\]

Let us suppose that $\alpha$ is algebraic and let us consider the linear forms
\begin{equation*}
L_1(X,Y,Z)=Z\alpha-X,\quad
L_2(X,Y,Z)=Z\alpha^{-1}-Y,\quad
L_3(X,Y,Z)=Z.
\end{equation*}

Then, since $a_n=a_1$ and $p_n=q_{n-1}$,
\begin{align*}
\prod\limits_{i=1}^3 | L_i(p_{n-1},q_n,q_{n-1}) |&=|q_{n-1}\alpha-p_{n-1}|\cdot|q_{n-1}\alpha^{-1}-q_{n}|\cdot|q_{n-1}|=\\
&=\frac{1}{|q_{n-1}\alpha_n|}\cdot\frac{1}{|q_{n-1}\alpha_{n+1}|}\cdot|q_{n-1}|=\frac{1}{|q_{n-1}\alpha_n\alpha_{n+1}|}<\\
&<\frac{1}{|q_{n-1}|}=\frac{|a_n|}{|q_{n}|}=\frac{|a_1|}{|q_{n}|}=\frac{|a_1||q_{n}|^{\varepsilon-1}}{|q_{n}|^{\varepsilon}}<\frac{1}{|q_{n}|^{\varepsilon}}=\\
&=\max\{|p_{n-1}|,|q_{n}|,|q_{n-1}|\}^{-\varepsilon},
\end{align*}
for $0<\varepsilon<1$ and $n$ sufficiently large. Since the absolute values of $p_{n-1},q_{n-1},q_{n}$ are unbounded, by Theorem~\ref{Thm: Ratliff} they lie in a finite union of proper linear subspaces of $K(T)^3$. This means that
\[up_{n-1}+vq_{n-1}+wq_n=0,\]
for some $u,v,w\in K[T]$, not all zero, and for infinitely many $n$, that is
\[u\frac{p_{n-1}}{q_{n-1}}+v+w\frac{q_n}{q_{n-1}}=0.\]
Recalling that
\begin{align*}
\frac{p_{n-1}}{q_{n-1}}&=\alpha-\frac{(-1)^{n-1}}{q_{n-1}(q_{n-1}\alpha_n+q_{n-2})}=\alpha-x_n,\\
\frac{q_{n}}{q_{n-1}}&=\alpha^{-1}-\frac{(-1)^{n+1}}{p_{n}(p_{n}\alpha_{n+1}+p_{n-1})}=\alpha^{-1}-y_n.
\end{align*}
We obtain
\[u\alpha+v+w\alpha^{-1}=ux_n+wy_n.\]
Letting $n\to+\infty$ and using
\[\lim_{n\to+\infty}x_n=\lim_{n\to+\infty}y_n=0,\]
we conclude that
\[u\alpha+v+w\alpha^{-1}=0.\]
Therefore,
\[u\alpha^2+v\alpha+w=0\]
and consequently $\alpha$ is quadratic or transcendental.
\end{proof}

\section{Future research}\label{sec: spunti}
In this section, we present a promising research direction that may lead to more general results.
First observe that, if $\alpha \in K((T^{-1}))$ is transcendental and $a/b \in K(T)$ with $a/b \ne 0$, then $\frac{a}{b}\alpha$ is also transcendental. Moreover, for continued fractions over function fields, one can derive recursively a continued fraction expansion of $\frac{a}{b}\alpha$.
In \cite{Gri} and \cite[Section~1.3.4]{Mala} the authors present a recursive algorithm for expressing the partial quotients of a continued fraction expansion of $\frac{a}{b} \alpha$ in terms of the partial quotients of $\alpha$, where $\alpha$ is a Laurent series and $a/b$ a rational function. 
In particular, Malagoli presents it in \cite[Theorem~1.3.26]{Mala}, together with examples and special cases, while Grisel presents and proves it in \cite[Theorem~3]{Gri}.
This provides an advantage over the real case, where finding the continued fraction expansion of $\frac{a}{b}\alpha$ from that of $\alpha$ is not straightforward.
In \cite[Remark~1.3.28]{Mala} the formula takes a particularly simple form when $a/b=(T- \lambda)$, where $\lambda \in K$, as shown in the following proposition.

\begin{Proposition}[\cite{Mala}]\label{prop: t-lambda}
Let $\alpha \in K((T^{-1}))$, with $\alpha = [a_0,a_1,a_2, \ldots]$ and let $a_n' \in K[T]$ for all $n \ge 0$ such that
\[a_n = (T- \lambda) a_n' + a_n(\lambda)\]
for all $n \geq 0$. Let $q_n \in K[T]$ be the denominator of the $n$-th convergent of $\alpha$. Then, if $q_n(\lambda) \ne 0$ for all $n \ge 0$, we have
\[(T- \lambda) \alpha = [b_0, b_1, b_2, \ldots]\]
where
\[\begin{cases}
b_0 = (T- \lambda)a_0, \\
b_{2n+1} = (-1)^nq_n(\lambda)^2a_{n+1}' &\quad \text{for all} \quad n \ge 0, \\
b_{2n} = (-1)^{n-1}(q_{n-1}q_{n})(\lambda)^{-1} (T - \lambda) &\quad \text{for all} \quad n \ge 1.
\end{cases}\]
\end{Proposition}

Using Proposition~\ref{prop: t-lambda}, we can construct an infinite family of transcendental continued fractions having partial quotients pairwise distinct and of bounded norm. Therefore, this family consists of continued fractions that are neither automatic, palindromic, quasi-periodic, nor of Liouville-type.

\begin{remark}\label{rem: distincpc}
Let $\alpha \in K((T^{-1}))$ ($\operatorname{char}(K)=0$) with
\[
\alpha=[a_0,a_1,a_2,\ldots],
\]
where $a_n\in\mathbb{Z}[T]$ for all $n\ge0$, each $a_n$ is monic, and
$\deg(a_n)\ge2$ for every $n\ge1$. Then, for every
$\lambda\in\mathbb{Q}\setminus\mathbb{Z}$, the partial quotients of $(T-\lambda)\alpha$ are pairwise distinct.

Indeed, since $a_n\in\mathbb{Z}[T]$, every denominator $q_n$ of the convergents
belongs to $\mathbb{Z}[T]$ and is monic. Hence, by the Rational Root Theorem,
$q_n(\lambda)\neq0$ for every $\lambda\in\mathbb{Q}\setminus\mathbb{Z}$.
Moreover, for $0\le n<m$, the polynomials
\[
q_{m-1}q_m\pm q_{n-1}q_n,\qquad
q_m^2\pm q_n^2,\qquad
q_m^2q_{n-1}q_n\pm1,\qquad
q_n^2q_{m-1}q_m\pm1
\]
are monic elements of $\mathbb{Z}[T]$. Therefore, they do not vanish at any
$\lambda\in\mathbb{Q}\setminus\mathbb{Z}$ by Rational Root Theorem.
It follows that the leading coefficients of the partial quotients of
$(T-\lambda)\alpha$ are pairwise distinct, and consequently the partial
quotients themselves are pairwise distinct.
Thus, starting from a transcendental continued fraction with bounded partial quotients, we obtain another transcendental continued fraction whose partial quotients remain bounded and are pairwise distinct.
\end{remark}
The results of this paper and the above discussion naturally lead to the following open problem.
\begin{Problem}
Characterize the classes of transcendental continued fractions obtained from those considered in this paper by multiplication by a rational function $a/b$, where $a, b \in K[T]$ and $b \ne 0$.
\end{Problem}

\end{document}